\documentclass[12pt]{amsart}
\usepackage{amssymb}
\usepackage{amsmath, amscd}
\usepackage{amsthm}
\usepackage{comment}
\usepackage[table,xcdraw]{xcolor}
\usepackage{ulem}
\usepackage{tikz}
\usepackage{float}
\usepackage{graphicx}
\usepackage{circuitikz}
\usetikzlibrary{positioning}
\usepackage[colorlinks=true, linkcolor=blue,urlcolor=blue]{hyperref}

\newtheorem{theorem}{Theorem}[section]

\newtheorem{proposition}[theorem]{Proposition}
\newtheorem{lemma}[theorem]{Lemma}

\newtheorem{corollary}[theorem]{Corollary}
\theoremstyle{definition}

\newtheorem{example}[theorem]{Example}
\newtheorem{definition}[theorem]{Definition}

\newtheorem{remark}[theorem]{Remark}

\newtheorem{problem}[theorem]{Problem}

\newcommand{\bigzero}{\mbox{\normalfont\Large\bfseries 0}}

\begin{document}

\author[P. Danchev]{Peter Danchev}
\address{Institute of Mathematics and Informatics, Bulgarian Academy of Sciences, 1113 Sofia, Bulgaria}
\email{danchev@math.bas.bg; pvdanchev@yahoo.com}
\author[A. Javan]{Arash Javan}
\address{Department of Mathematics, Tarbiat Modares University, 14115-111 Tehran Jalal AleAhmad Nasr, Iran}
\email{a.darajavan@modares.ac.ir; a.darajavan@gmail.com}
\author[O. Hasanzadeh]{Omid Hasanzadeh}
\address{Department of Mathematics, Tarbiat Modares University, 14115-111 Tehran Jalal AleAhmad Nasr, Iran}
\email{o.hasanzade@modares.ac.ir; hasanzadeomiid@gmail.com}
\author[A. Moussavi]{Ahmad Moussavi}
\address{Department of Mathematics, Tarbiat Modares University, 14115-111 Tehran Jalal AleAhmad Nasr, Iran}
\email{moussavi.a@modares.ac.ir; moussavi.a@gmail.com}

\title[C$\Delta$ rings]{Rings in which all elements are the sum of \\ a central element and an element from $\Delta (R)$}

\maketitle




\begin{abstract} We define and consider in-depth the so-called "{\it $C\Delta$ rings}" as those rings $R$ whose elements are a sum of an element in $C(R)$ and of an element in $\Delta(R)$. Our achieved results somewhat strengthen these recently obtained by Ma-Wang-Leroy in Czechoslovak Math. J. (2024) as well as these due to Kurtulmaz-Halicioglu-Harmanci-Chen in Bull. Belg. Math. Soc. Simon Stevin (2019).

Specifically, we succeeded to establish that exchange $C\Delta$ rings are always clean as well as that exchange CN rings are strongly clean. Likewise, we prove that, for any ring $R$, the ring of formal power series $R[[x]]$ over $R$ is $C\Delta$ if, and only if, so is $R$. And, furthermore, we show that, for any ring $R$, if the polynomial ring $R[x]$ is a $C\Delta$ ring, then $R$ satisfies the K\"othe conjecture. Some other closely related things concerning certain extensions of $C\Delta$ rings are also presented.

\end{abstract}

\section{Introduction and Basic Concepts}

In the current paper, let throughout $R$ denote an associative ring with identity element, which ring is {\it not} necessarily commutative. Typically, for such a ring $R$, the sets $U(R)$, $Nil(R)$ and $C(R)$ represent the set of invertible elements (i.e., the unit group of $R$), the set of nilpotent elements and the set of central elements in $R$, respectively. Additionally, $J(R)$ denotes the Jacobson radical of $R$. Likewise, the whole ring of \( n \times n \) matrices over \( R \) and the ring of \( n \times n \) upper triangular matrices over \( R \) are, respectively, designed by \( {\rm M}_n(R) \) and \( {\rm T}_n(R) \).\\

The key instrument of the present article unambiguously plays the set $\Delta(R)$, which was introduced by Lam in \cite [Exercise 4.24]{10} and was recently examined by Leroy-Matczuk in \cite{2}. As pointed out by the authors in \cite [Theorem 3 and 6]{2}, $\Delta(R)$ is the largest Jacobson radical's subring of $R$ that is closed with respect to multiplication by all units (as well as by quasi-invertible elements) of $R$. Also, the containment $J(R) \subseteq \Delta(R)$ is fulfilled always. Moreover, $\Delta(R)=J(T)$, where $T$ is the subring of $R$ generated by all units of $R$, and the equality $\Delta(R)=J(R)$ holds if, and only if, $\Delta(R)$ is an ideal of $R$. Explicitly, an element $a$ in a ring $R$ is from $\Delta(R)$ if $1-ua$ is a unit for all $u\in U(R)$.\\

Furthermore, a ring $R$ is said to be $CU$ if each element $a\in R$ has a decomposition $a=c+u$, where $c$ is a central element and $u$ is a unit -- to be exact, this notion was introduced by Kurtulmaz et al. in \cite{6}. They showed that, if $\mathbb F$ is a field which is {\it not} isomorphic to $\mathbb Z_2$, then $M_2(\mathbb F)$ is a $CU$ ring.

In the other vein, a ring $R$ is called $CN$ if each element $a$ of $R$ has a decomposition $a=c+n$, where $c$ is a central element and $n$ is a nilpotent element. This concept was introduced in \cite{11}. They showed that every $CN$ ring is $CU$, but the converse is definitely {\it not} true in general. They also characterized all elements in $M_n(R)$ having a $CN$ decomposition. They, moreover, proved for a division ring $R$ that, if $M_n(R)$ is a $CN$ ring, then the cardinality of $C(R)$ is strictly greater than $n$.\\

Similarly, an element $a\in R$ has a $CJ$ decomposition if $a=c+j$ for some $c\in C(R)$ and $j\in J(R)$ and a ring $R$ is called $CJ$ if every element of $R$ has a $CJ$ decomposition. This notion was introduced by Ma et al. in \cite{9}. They study the behavior of this class of rings under various ring extensions. In particular, they proved that, for any ring $R$, the subring $C(R)+J(R)$ is always a $CJ$ ring, and that if $R[x]$ is a $CJ$ ring, then $R$ satisfies the Köthe conjecture.\\

So, as a non-trivial extension of $CJ$ rings, we are motivated to introduce and investigate the new and more expanded concept of a {\it $C\Delta$ ring}, which will be the leitmotif of our study, like this: a ring $R$ is called $C\Delta$ if each element $a\in R$ has a decomposition $a=c+r$, where $c\in C(R)$ and $r\in \Delta(R)$. Apparently, all $CJ$ rings are $C\Delta$, but the opposite implication does {\it not} hold in all generality as we will illustrate below. \\

In the following, we give definitions of some other subrings of \( {\rm T}_n(R) \) in order to discuss in the sequel whether they satisfy the $C\Delta$ property:
\[
D_n(R) = \left\{ (a_{ij}) \in M_n(R) \ \middle| \ \text{all diagonal entries of } (a_{ij}) \text{ are equal} \right\};
\]
\[
V_n(R) = \left\{ \sum_{i=j}^{n} \sum_{j=1}^{n} a_j e_{(i-j+1)i} \ \middle| \ a_j \in R \right\};
\]

\[
V_n^k(R) = \left\{ \sum_{i=j}^{n} \sum_{j=1}^{k} x_j e_{(i-j+1)i} + \sum_{i=j}^{n-k} \sum_{j=1}^{n-k} a_{ij} e_{j(i+k+1)} \ \middle| \ x_j, a_{ij} \in R \right\},
\]
where \( x_{ij} \in R, a_{js} \in R, \ 1 \leq i \leq k, \ 1 \leq j \leq n - k \ \text{and} \ k + 1 \leq s \leq n \);
\[
D_n^k(R) = \left\{ \sum_{i=1}^{k} \sum_{j=k+1}^{n} a_{ij} e_{ij} + \sum_{j=k+2}^{n} b_{(k+1)j} e_{(k+1)j} + c I_n \ \middle| \ a_{ij}, b_{ij}, c \in R \right\},
\]
where \( k = [\frac{n}{2}] \), i.e., satisfies \( n = 2k \) when \( n \) is an even integer and \( n = 2k + 1 \) when \( n \) is an odd integer.

\medskip

Our principal work in what follows is planned thus: In the next section, we are concerned with the exploration of some crucial structural properties of $C\Delta$ rings arising under various constructions, as well as we will discuss certain connections between the aforementioned classes of rings (see, e.g., Corollaries~\ref{exch} and \ref{4.7}). In the subsequent third section, we are focussed on the possible extensions of $C\Delta$ rings and again some transversal between the existing classes of rings is obtained (see, e.g., Theorems~\ref{4.10}, \ref{4.11} and \ref{3.9}). In the final fourth section, we pose two queries of some interest and importance which, hopefully, could stimulate a further research on this subject (see Problems~\ref{p1} and \ref{p2}).

\section{Basic properties of $C\Delta$ rings}

In this section, we introduce the concept of a $C\Delta$ ring and investigate its elementary properties. We also give relations between $C\Delta$ rings and some other related rings. For this purpose, we need a series of preliminary technicalities.

\begin{lemma}\label{2.2}
For any ring \(R\), the following equality is true:
	\[
	U(R) + \Delta(R) = U(R).
	\]
\end{lemma}

\begin{proof}
We know that \(U(R) \subseteq U(R) + J(R) \subseteq U(R) + \Delta(R)\). To treat the reciprocal, given \(x \in U(R) + \Delta(R)\), we may write \(x = u+r\), where \(u \in U(R)\) and \(r \in \Delta(R)\). One easily checks that \(x = u + r = u(1+u^{-1}r) \in U(R) \), as required.
\end{proof}

\begin{lemma}\label{2.3}
Let \(R\) be a ring. For \(n \geq 2\),
	\[ \Delta(M_n(R)) = J(M_n(R)) = M_n(J(R)).
	\]
\end{lemma}

\begin{proof}
Knowing that, for all \(n \geq 2\), $M_n(R)$ is generated by its units, we apply \cite [Corollary 4]{2} to get that $\Delta(M_n(R)) = J(M_n(R))$, as stated.	
\end{proof}

\begin{lemma}\label{2.7}
Let \( R \) be a ring. If \( I \subseteq J(R) \) and \( \Delta(R/I) = 0 \), then \( \Delta(R)=I = J(R)\).
\end{lemma}

\begin{proof}
Since \( J(R) \subseteq \Delta(R) \), it suffices to show that \( \Delta(R) \subseteq I \). Choose $x\in \Delta(R)$ and $x\notin I$. So, $x+I$ is a non-zero element of $R/I$ and hence $x+I\notin \Delta(R/I)$. Therefore, by definition of $\Delta$, there exists $y\in R$ such that $(1+xy)+I$ is not unit in $R/I$, but this is a contradiction. Then, \( \Delta(R) \subseteq I \) whence \( \Delta(R)=I = J(R)\), as asserted.
\end{proof}

Two automatic consequences hold:

\begin{corollary}\label{2.8}
For any ring \( R \), \( \Delta(R) = J(R) \) if, and only if, \( \Delta(R/J(R))= (0) \).
\end{corollary}

\begin{corollary}\label{2.9}
Let \( R \) be a ring and \( I \subseteq J(R) \). Then, $ \dfrac {I+\Delta(R)}{I} = \Delta(R/I)$.
\end{corollary}

Let us recollect once again our key machinery explicitly in a more expanded form.

\begin{definition}
Let $R$ be a ring. An element $a\in R$ is called $C\Delta$ or just it has a {\it $C\Delta$-decomposition} if $a=c+r$, where $c\in C(R)$ and $r\in \Delta(R)$. If every element of $R$ has such a $C\Delta$-decomposition, then $R$ is called a {\it $C\Delta$ ring}.	
\end{definition}

The following appeared to be obviously true:

\begin{example}\label{2.10}
The following are fulfilled:	
\begin{enumerate}
\item
Every commutative ring is $C\Delta$.
\item
Every radical ring $R$, i.e., $R = J(R)$, is $C\Delta$.
\item
Every $\Delta$-ring $R$, i.e., $R = \Delta(R)$, is $C\Delta$.
\item
Every CJ ring is $C\Delta$.
\end{enumerate}
\end{example}

Recall that a ring $R$ is said to be {\it Dedekind-finite} if $ab=1$ implies $ba=1$ for any $a, b\in R$. In other words, all one-sided inverses in the ring are two-sided. Also, a ring $R$ is called {\it reduced} if it contains no non-zero nilpotent elements.

\medskip

We now proceed by proving the following statement.

\begin{proposition}\label{4.3}
Let $R$ be a $C\Delta$ ring. Then, the following are valid:	
\begin{enumerate}
\item
For any ideal $I\subseteq J(R)$, $R/I$ is $C\Delta$ (in particular, $R/J(R)$ is $C\Delta$).
\item
For any $a, b\in R$, $ab-ba\in \Delta(R)$.
\item
$R$ is Dedekind-finite.
\item
If $a^2\in \Delta(R)$, then $a\in \Delta(R)$.
\item
$R$ is a $CU$ ring.
\item
Always $Nil(R)\subseteq J(R)$ and $Nil(R)\subseteq \Delta(R)$.
\item
If $a^2\in J(R)$, then $a\in J(R)$ (in particular, $R/J(R)$ is reduced).
\item
For any $e\in Id(R)$, the corner subring $eRe$ is a $C\Delta$ ring.
\end{enumerate}
\end{proposition}

\begin{proof}
\begin{enumerate}
\item
It is straightforward and, therefore, we omit the details.
\item
Suppose $a=c+d$ and $b=c^{\prime}+d^{\prime}$, where $c, c^{\prime} \in C(R)$ and $d, d^{\prime}\in \Delta(R)$. Thus, we have $ab-ba=dd^{\prime}-d^{\prime}d\in \Delta(R)$.
\item
Let $a, b\in R$. So, $ab-ba\in \Delta(R)$ by (ii). If $ab=1$, then $1-ba\in \Delta(R)$. Hence, $ba\in 1-\Delta(R)\subseteq U(R)$.	
\item
Suppose $a=c+d$, where $c\in C(R)$ and $d\in \Delta(R)$. Thus, $a^2=c^2+2cd+d^2$. On the other side, since $a^2\in \Delta(R)$, it must be that $(1-a)(1+a)=1-a^2\in U(R)$ and hence $1+a\in U(R)$. So, $$(1+a)-d=c+1\in U(R)+\Delta(R)\subseteq U(R).$$ Therefore, $(c+1)d\in \Delta(R)$ whence $cd+d\in \Delta(R)$. But, since $\Delta(R)$ is a subring of $R$, we derive $cd\in \Delta(R)$ and so $c^2=a^2-2cd-d^2 \in \Delta(R)$. Now, we show that $c\in \Delta(R)$. Given $u\in U(R)$, we may write $(1-uc)(1+uc)=1-u^2c^2\in U(R)$ yielding $1-uc\in U(R)$ and $c\in \Delta(R)$. Consequently, $a=c+d\in \Delta(R)$.
\item
Choose $a\in R$. Then, $a-1=c+d$, where $c\in C(R)$ and $d\in \Delta(R)$. So, $a=c+(1+d)$, where $1+d\in 1+\Delta(R)\subseteq U(R)$. Hence, $R$ is $CU$.
\item
It is pretty clear, so we leave the arguments.
\item
Using (iv), we have $a\in \Delta(R)$. Take $r\in R$ and write $r=c+d$, where $c\in C(R)$ and $d\in \Delta(R)$. So, we obtain $(1-ra)=1-(c+d)a=1-ca-da$. Now, we show that $1-ca\in U(R)$. In fact, $(1-ca)(1+ca)=1-c^2a^2\in U(R)$ (note that $a^2\in J(R)$) and hence $1-ca\in U(R)$. Therefore, $$1-ra=(1-ca)-da\in U(R)+\Delta(R)\subseteq U(R).$$ This shows that $a\in J(R)$.
\item
Choosing $a\in eRe$, we write $a=c+d$, where $c\in C(R)$ and $d\in \Delta(R)$. Since $a\in eRe$, we deduce $a=ec+ed=ce+de$ forcing that $ed=de$. Evidently, $ec\in C(eRe)$. Now, we prove that $ed=de\in \Delta(eRe)$. For this purpose, assume that $u\in U(eRe)$, so $u^{\prime}=u+(1-e)\in U(R)$ and, moreover, $v=1-u^{\prime}d\in U(R)$. But, $eve=ev=ve=e-ued$ and, also, $(eve)(ev^{-1}e)=e$. Therefore, $ev=e-ued\in U(R)$ and hence $ed=de\in \Delta(eRe)$.
\end{enumerate}
\end{proof}

Recall that a ring $R$ is {\it exchange} if, for each $a\in R$, there exists $e^2=e\in aR$ such that $(1-e)\in (1-a)R$ (see \cite{nic}).

\begin{corollary}\label{exch}
Let \( R \) be an exchange $C\Delta$ ring. Then, $R$ is clean.
\end{corollary}

\begin{proof}
Since $R$ is exchange, $R/J(R)$ is exchange and all idempotents lift modulo $J(R)$ in view of \cite[Proposition 1.5]{nic}. Likewise, $R/J(R)$ is reduced utilizing Proposition \ref {4.3} (vii), whence it is abelian. Therefore, $R/J(R)$ is abelian exchange, and hence it is clean. Finally, taking into account \cite{nic}, one concludes that $R$ is clean, as required.
\end{proof}	

\begin{proposition}\label{2.14}
Let \(R\) be a ring. Then, a $C\Delta$ ring \(R\) is uniquely $C\Delta$ (that is, each element has a unique $C\Delta$-decomposition) if, and only if, \(\Delta(R) \cap C(R) = \{0\}\).
\end{proposition}

\begin{proof}
Suppose that $R$ is uniquely $C\Delta$ and that \(c \in C \cap \Delta\). Let \(0 \not = r \in \Delta(R)\). We have \(1+r = (1-c) + (c+r)\), where \(1-c \in C(R)\) and \(c+r \in \Delta(R)\), because \(\Delta(R)\) is a subring of \(R\). The uniqueness of the \(C \Delta\)-decomposition gives that \(c=0\).

Conversely, if \(C \cap \Delta = 0\) and \(c+r=c^\prime + r^\prime\) are two \(C\Delta\) decompositions, we infer \(c-c^\prime=r^\prime-r \in C \cap \Delta = 0\) and hence \(c = c^\prime, r = r^\prime\), as required.
\end{proof}

\begin{proposition}\label{2.15}
Let \(R\) be a ring. Then, the following statements hold:
\begin{enumerate}
\item
\(U(R) \cap (C + \Delta) \subseteq U(C + \Delta)\).
\item
 \(\Delta(R) \subseteq \Delta(C + \Delta)\).
\item
\(C(R) \subseteq C(C + \Delta)\).
\end{enumerate}
\end{proposition}

\begin{proof}
(i) Choose \(a \in U(R) \cap (C + \Delta)\). Then, \(a = c + r\) for \(c \in C(R)\), \(r \in \Delta(R)\), and there exists \(b \in U(R)\) such that \(ab=ba=1\). We deduce that \(c = a - r \in U(R) + \Delta(R) \subseteq U(R)\) whence \(c \in C(R) \cap U(R) \subseteq U(C(R))\). Therefore, \(ab = (c+r)b = cb + rb = 1\). Thus, \[b = c^{-1}(1-rb) = c^{-1} - c^{-1}rb \in C + \Delta,\] because $\Delta(R)$ is a subring of $R$. This manifestly shows that the inverse of \(a\) is, indeed, in \(C + \Delta\).\\
(ii) Choose \(x \in \Delta(R)\). So, \(x \in C + \Delta\) and \(1+xy \in U(R)\) for every \(y \in U(R)\). Thus, \(1 + xy \in U(R)\) for every \(y \in U(C + \Delta)\), and there exists \(v \in U(R)\) such that \((1+xy)v = v(1+xy)=1\). Therefore, we have \(v = 1 -xyv \in C + \Delta\), and hence \(1+xy \in U(C + \Delta)\) for every \(y \in U(C + \Delta)\). This demonstrates that \(x \in \Delta(C + \Delta)\).\\
(iii) Choose \(x \in C(R)\). Then, \(x \in C + \Delta\) and \(xy = yx\) for every \(y \in R\). Sos, \(xy = yx\) for every \(y \in C + \Delta\). This gives that \(x \in C(C + \Delta)\).
\end{proof}

The next appears immediately.

\begin{corollary}\label{2.16}
Let \(R\) be a ring. Then, \(C(R) + \Delta(R)\) is a \(C\Delta\) ring.
\end{corollary}

Recall that a ring $R$ is called {\it division} if every non-zero element of $R$ is invertible. Also, a ring $R$ is called {\it local} if $R/J(R)$ is a division ring.

\begin{example}\label{3.24}
(i) It is well known that any division ring is a local ring, and hence it is a $CU$ ring in virtue of \cite[Example 2.2]{6}. But, one plainly verifies that division rings are not $C\Delta$.\\
(ii) Letting \(K\) be a division ring, we consider the ring \(D_2 (K)\). We know that \(D_2(K)\) is a local ring, and so it is a \(CU\) ring, but a simple check guarantees that it is not a \(C\Delta\) ring.
\end{example}

A ring $R$ is called $\Delta U$, provided $1+\Delta(R) = U(R)$ (see \cite {3}).

\begin{remark}\label{2.18}
One readily verifies that, if \(R\) is a $\Delta U$ ring, then \(R\) is \(C\Delta\) precisely when it is \(CU\).
\end{remark}

\begin{proposition}\label{2.19}
Let \(R\) be a ring and \(a \in R\). Then, \(a\) has a \(C\Delta\)-decomposition if, and only if, for each \(p \in U(R)\), \(pap^{-1}\) has a \(C\Delta\)-decomposition.
\end{proposition}

\begin{proof}
Assume that \(a\) has a \(C\Delta\)-decomposition. So, \(a = c+r\), where \(c \in C(R)\) and \(r \in \Delta(R)\) such that, for an invertible \(p \in R\), \(pap^{-1} = pcp^{-1} + prp^{-1}\), where \(pcp^{-1} = c\) is central and \(prp^{-1} \in \Delta(R)\).

Conversely, supposing that \(pap^{-1} \in R\) has a \(C\Delta\)-decomposition, we receive \(pap^{-1} = t + x\), where \(t \in C(R)\) and \(x \in \Delta(R)\). Hence, \(a = p^{-1}tp + p^{-1}xp\) is a \(C\Delta\)-decomposition of \(a\).
\end{proof}

A ring $R$ is said to be $UJ$, provided $U(R)=1+J(R)$ (for more details, see \cite{12}). Recall that a ring $R$ is called $UU$, provided $U(R)=1+Nil(R)$ (see, for a more account, \cite {13}).

\begin{lemma}\label{2.22}
The following items hold:
\begin{enumerate}
\item
If \( R \) is a \(UJ\) ring, then \( Nil(R) \subseteq J(R) \).
\item
If \( R \) is a \(UU\) ring, then \( J(R), \Delta(R) \subseteq Nil(R) \).
\item
If \( R \) is a $\Delta$U ring, then \(Nil(R) \subseteq \Delta(R) \).
\end{enumerate}
\end{lemma}

\begin{proof}
\begin{enumerate}
\item
Choosing $a\in Nil(R)$, we get $1+a\in 1+Nil(R)\subseteq U(R)=1+J(R)$. Hence, $a\in J(R)$.
\item
Choosing $a\in J(R)$, we get $1+a\in 1+J(R)\subseteq U(R)=1+Nil(R)$. Hence, $a\in Nil(R)$.\\
Since, $1+\Delta(R)\subseteq U(R)$, we are done.
\item
It can be argued in the same way as above, so the details are dropped off.
\end{enumerate}	
\end{proof}

\begin{corollary}\label{2.23}
Let $R$ be a ring. The following points hold:
\begin{enumerate}
\item
If \( R \) is \(UU\) and \(C\Delta\), then it is \(CN\).
\item
If \(R\) is \(\Delta U\) and \(CN\), then it is \(C \Delta\).
\item
If \(R\) is $UJ$ and \(CN\), then it is \(CJ\).
\item
If \( R \) is \(UU\) and \(CJ\), then it is \(CN\).
\end{enumerate}
\end{corollary}

\begin{lemma}\label{4.1}
Let \( R \) be a commutative ring. Then, the following are equivalent:
\begin{enumerate}
\item
\( R \) is a $CN$ ring.
\item
\( R \) is a $CJ$ ring.
\item
\( R \) is a $C\Delta$ ring.
\end{enumerate}
\end{lemma}

\begin{proof}
\((i) \Rightarrow (ii)\). This is obvious, because, in commutative rings, we always have $Nil(R)\subseteq J(R)$.
	
\((ii) \Rightarrow (iii)\). It is simple, because we always have $J(R)\subseteq \Delta(R)$.
	
\((iii) \Rightarrow (i)\). It is immediate by referring to \cite[Example 2.2]{11}.
\end{proof}

A ring $R$ is said to be {\it semi-local} if $R/J(R)$ is a left Artinian ring or, equivalently, $R/J(R)$ is a semi-simple ring.

\begin{lemma}\label{4.2}
Let \( R \) be an Artinian ring. Then, the following are equivalent:
\begin{enumerate}
\item
\( R \) is a $CN$ ring.
\item
\( R \) is a $CJ$ ring.
\item
\( R \) is a $C\Delta$ ring.
\end{enumerate}
\end{lemma}

\begin{proof}
Since every Artinian ring is known to be semi-local, we have $J(R)=\Delta(R)$ by virtue of \cite[Theorem 11]{2}, whence (ii) and (iii) are equivalent.
	
\((i) \Rightarrow (ii)\). It suffices to show that $Nil(R)\subseteq J(R)$. Since $R$ is Artinian, one readily writes that $R/J(R)\cong \prod M_{n_i}(D_i)$ for some matrix rings $M_{n_i}(D_i)$ over division rings $D_i$. Also, $R/J(R)$ is $CN$ invoking \cite[Lemma 2.15]{11}, and hence each $M_{n_i}(D_i)$ is $CN$ consulting with \cite[Proposition 2.16]{11}. Since we know that the matrix rings are not $CN$, we arrive at $n_i=1$ for every $i$. Finally, $R/J(R)\cong \prod D_i$ and this insures that $R/J(R)$ is reduced. Thus, $Nil(R)\subseteq J(R)$.

\((ii) \Rightarrow (i)\). This is straightforward, because in Artinian rings we have always $J(R)\subseteq Nil(R)$.
\end{proof}

As finite rings are always Artinian, we directly extract the following consequence.

\begin{corollary}\label{4.4}
Let \( R \) be a finite ring. Then, the following are equivalent:
\begin{enumerate}
\item
\( R \) is a $CN$ ring.
\item
\( R \) is a $CJ$ ring.
\item
\( R \) is a $C\Delta$ ring.
\end{enumerate}
\end{corollary}

Let $Nil_{*}(R)$ denote the {\it prime radical} (or, in other terms, the {\it lower nil-radical}) of a ring $R$, i.e., the intersection of all {\it prime ideals} of $R$. We know that $Nil_{*}(R)$ is a nil-ideal of $R$. It is also long known that a ring $R$ is called {\it $2$-primal} if its lower nil-radical $Nil_{*}(R)$ consists precisely of all the nilpotent elements of $R$, that is, $Nil_{*}(R)=Nil(R)$. For instance, it is well known that reduced rings and commutative rings are both $2$-primal.

\medskip

For an endomorphism $\alpha$ of a ring $R$, $R$ is called {\it $\alpha$-compatible} if, for any $a,b\in R$, $ab=0\Longleftrightarrow a\alpha (b)=0$, and in this case $\alpha$ is clearly injective.

\medskip

Let $R$ be a ring and suppose $\alpha : R \to R$ is a ring endomorphism. Then, $R[x; \alpha]$ denotes the {\it ring of skew polynomial} over $R$ with multiplication defined by $xr = \alpha(r)x$ for all $r \in R$. In particular, $R[x] = R[x; 1_R]$ is the {\it ring of polynomials} over $R$.

\begin{proposition}\label{4.12}
Let \( R \) be a 2-primal and \(\alpha\)-compatible ring. Then,
\[
\Delta(R[x, \alpha]) = \Delta(R) + Nil_{*}(R[x, \alpha])x.
\]
\end{proposition}

\begin{proof}
Writing that \( f = \sum_{i=0}^{n} a_i x^i \in \Delta(R[x, \alpha]) \), then, for every \( u \in U(R) \), we have that \( 1 - uf \in U(R[x, \alpha]) \). Thus, an appeal to \cite[Corollary 2.14]{18} ensures that \( 1 - ua_0 \in U(R) \) and, for every \( 1 \le  i \le  n \), \( ua_i \in Nil_{*}(R) \). Since \( Nil_{*}(R) \) is an ideal, we have \( a_0 \in \Delta(R) \) and, for every \( 1 \le  i \le  n \), \( a_i \in Nil_{*}(R) \). But, as \( R \) is $2$-primal, \cite[Lemma 2.2]{18} is applicable to get that \( Nil_{*}(R)[x,\alpha] = Nil_{*}(R[x, \alpha]) \).
	
Conversely, assume \( f \in \Delta(R) + Nil_{*}(R[x, \alpha])x \) and \( u \in U(R[x,\alpha]) \). Then, a consultation with \cite[Corollary 2.14]{18} assures that \( u \in U(R) + Nil_{*}(R[x, \alpha])x \). Since \( R \) is a $2$-primal ring, we have \[ 1 - uf \in U(R) + Nil_{*}(R[x, \alpha])x \subseteq U(R[x, \alpha]) ,\] and thus \( f \in \Delta(R[x, \alpha]) \), as required.
\end{proof}

\begin{lemma}\label{4.9}
Let \( R \) be a $CN$ ring. Then, for every $e\in Id(R)$, $eRe$ is too a $CN$ ring.
\end{lemma}

\begin{proof}
Letting $a\in eRe$, we then write $a=c+n$, where $c\in C(R)$ and $n\in Nil(R)$. Since $a\in eRe$, we have $a=ec+en=ce+ne$ and this, in turn, forces that $en=ne$. But, $ec\in C(eRe)$ and we easily conclude that $en=ne\in Nil(eRe)$. Finally, $eRe$ is a $CN$ ring, as claimed.
\end{proof}

A ring $R$ is called {\it semi-potent} if every one-sided ideal not contained in $J(R)$ contains a non-zero idempotent.

\begin{lemma}\label{4.5}
Let \( R \) be a semi-potent ring. Then, every $CN$ ring is $CJ$.
\end{lemma}

\begin{proof}
It suffices to show that $Nil(R)\subseteq J(R)$. For this target, we establish that $\overline{R} = R/J(R)$ is reduced. Since $R$ is a semi-potent ring, we know that $\overline{R}$ is also a semi-potent (and even a potent) ring.

Assume now that $0 \neq x \in \overline{R}$ with $x^2 = 0$. Thus, \cite[Theorem 2.1]{16} applies to get that there exists $e \in \overline{R}$ such that $e\overline{R}e \cong M_2(S)$, where $S$ is a non-zero ring. However, since $\overline{R}$ is a $CN$ ring in accordance with Lemma \ref{4.9}), $e\overline{R}e$ is also $CN$. However, on the other hand, it is known that $M_2(S)$ is not a $CN$ ring, which leads to a contradiction. Therefore, $\overline{R}$ is a reduced ring, as promised.
\end{proof}

\begin{corollary}\label{4.8}
Let \( R \) be a semi-potent $CN$ ring. Then, $R/J(R)$ is commutative.
\end{corollary}

\begin{proof}
With Lemma \ref{4.5} at hand, we conclude that $\overline{R} = R/J(R)$ is reduced and hence $Nil(\overline{R})=(0)$. Since $R$ is $CN$, it follows that $\overline{R}$ is $CN$ as well, and hence it is commutative.	
\end{proof}

\begin{lemma}\label{4.6}
Let \( R \) be a $CN$ ring. Then, $R$ is abelian.
\end{lemma}

\begin{proof}
Suppose $e\in Id(R)$. So, $e=c+n$, where $c\in C(R)$ and $n\in Nil(R)$. Therefore, we write $c-c^2=n-n^2-2cn\in Nil(R)$. Thus, owing to \cite[Lemma 3.5]{15}, there exists a monic polynomial $f(t) \in \mathbb{Z}[t]$, where ${f(a)}^2=f(a)$ and $m=c-f\in Nil(R)$. Since $c\in C(R)$, we have $f, m \in C(R)$. So, $e-f=m+n\in Nil(R)$. On the other side, $(e-f)^3=e-f$. This means that $e=f\in C(R)$, whence $R$ is abelian, as stated.
\end{proof}

As a final consequence, we detect:

\begin{corollary}\label{4.7}
Let \( R \) be an exchange $CN$ ring. Then, $R$ is strongly clean.
\end{corollary}

\section{Some extensions of $C\Delta$ rings}

We are here concentrated on certain non-trivial extensions of \(C\Delta\) rings of matrix type. Before stating and proving the main results, we deal with a few preliminaries.

\begin{lemma}\label{2.20}
Let \(R\) be a commutative ring. For any \(n \geq 2\), \(A \in M_n(R)\) has a \(C\Delta\)-decomposition if, and only if, \(A - cI_n \in M_n(J(R))\) for some \(c \in R\).
\end{lemma}

\begin{proof}
Assume that \(A \in M_n(R)\) has a \(C\Delta\)-decomposition. By assumption, there exists \(c \in R\) such that \(A - cI_n \in M_n(J(R))\). Note that, for any \(B \in M_n(R)\), \(B\) is central in \(M_n(R)\) precisely when there exists \(c \in R\) such that \(B = cI_n\), as required.

Conversely, suppose that, for any \(A \in M_n(R)\), there exists \(c \in R\) such that \(A - cI_n \in M_n(J(R))\). As \(cI_n\) is central in \(M_n(R)\), it follows that \(A \in M_n(R)\) has a \(C\Delta\)-decomposition, as needed.
\end{proof}

\begin{corollary}\label{2.21}
Let \(R\) be a ring and \(n\) a positive integer. Then, \(A \in M_n(R)\) has a \(C\Delta\)-decomposition if, and only if, for each \(p \in GL_n(R)\), \(pAp^{-1} \in M_n(R)\) has a \(C\Delta\)-decomposition.
\end{corollary}

\begin{remark}
Since \( \Delta(M_n(R)) = J(M_n(R)) \) for any \( n \geq 2 \) and \( M_n(R)/J(M_n(R)) \cong M_n(R/J(R)) \), we notice that \( M_n(R) \) is never a \(C \Delta\) ring for any \( n \geq 2 \).
\end{remark}

\begin{proposition}\label{2.25}
Let \( R \) be a commutative ring. Then, \( T_2(R) \) is a \(C\Delta\) ring if, and only if, for any \( a, b \in R \), there exists \( c \in R \) such that \( a - c, b - c \in \Delta(R) \).
\end{proposition}

\begin{proof}
Let \( A = \begin{pmatrix} a & 0 \\ 0 & b \end{pmatrix} \in T_2(R) \). So, $A$ has a \(C\Delta\)-decomposition if, and only if, there exist \( C = \begin{pmatrix} c & 0 \\ 0 & c \end{pmatrix} \in C(T_2(R)) \) and \( D = \begin{pmatrix} x & y \\ 0 & z \end{pmatrix} \in \Delta (T_2(R)) \) such that \( A = C + D \). Thus, \( D \in \Delta(T_2(R)) \) if, and only if, \( x, z \in \Delta(R) \). Also, \( A = C + D \) is a \(C\Delta\)-decomposition if, and only if, there exists \( c \in R \) such that \( A - cI_2 \in T_2(R)\) if, and only if, \( a - c, b - c \in \Delta(R) \), as asked for.
\end{proof}

\begin{example}\label{2.26}
Consider the ring \( R = \mathbb{Z} \). So, \( R \) is \(C\Delta\), but \( T_2(R) \) is not \(C\Delta\). Consider the ring \( A = \begin{pmatrix} 4 & 0 \\ 0 & 5 \end{pmatrix} \in T_2(R) \). Then, one inspects that there is no \( c \in R \) such that both \( 4 - c \) and \( 5 - c \) belong to \( \Delta(R) \). So, we deduce that \( T_2(R) \) is not \(C\Delta\), as pursued.
\end{example}

Let $R$ be a ring and suppose $\alpha : R \to R$ is a ring endomorphism. Standardly, $R[[x; \alpha]]$ denotes the ring of {\it skew formal power series} over $R$; that is, all formal power series in $x$ with coefficients from $R$ and with multiplication defined by $xr = \alpha(r)x$ for all $r \in R$. In particular, $R[[x]] = R[[x; 1_R]]$ is the {\it ring of formal power series} over $R$.

\medskip

We now come to the following criterion.

\begin{proposition}\label{2.13}
Let $R$ be a ring. The ring $R[[x; \alpha]]$ is $C\Delta$ if, and only if, $R$ is $C\Delta$.
\end{proposition}

\begin{proof}
Let $R[[x; \alpha]]$ be $C\Delta$. Consider $I:= R[[x; \alpha]]x$. Then, one verifies that $I$ is an ideal of $R[[x; \alpha]]$. Note also that a routine calculation leads to $J(R[[x; \alpha]])=J(R)+I$, so $I\subseteq J(R[[x; \alpha]])$. Since $R[[x; \alpha]]/I\cong R$, we derive with the aid of Proposition \ref{4.3}(i) that $R$ is $C\Delta$.

Conversely, let $R$ be $C\Delta$ and write $f(x) = a_0 + a_1 x + a_2 x + \dots \in R[[x; \alpha]]$. Then, $a_i=c_i+r_i$, where $c_i\in C(R)$ and $r_i\in \Delta(R)$ for every $i$. Thus, we have $f(x)=g(x)+k(x)$, where $$g(x)=c_0 + c_1 x + c_2 x + \dots \in C(R)[[x; \alpha]] \subseteq C(R[[x; \alpha]])$$ and $$k(x)=r_0 + r_1 x + r_2 x + \dots \in \Delta(R)[[x; \alpha]]=\Delta(R[[x; \alpha]]).$$ This enables us that $R[[x; \alpha]]$ is $C\Delta$, as expected.
\end{proof}

Our next automatic consequence sounds like this.

\begin{corollary}
A ring $R[[x]]$ is $C\Delta$ if, and only if, so is $R$.
\end{corollary}

We now have all the ingredients necessary to attack the truthfulness of the following main statement.

\begin{theorem}\label{4.10}
Let $R$ be a $2$-primal ring and let $\alpha$ be an endomorphism of $R$ such that $R$ is $\alpha$-compatible. The following are equivalent:
\begin{enumerate}
\item
$R$ is a $CN$ ring.
\item
$R[x; \alpha]$ is a $C\Delta$ ring.
\item
$R[x; \alpha]$ is a $CJ$ ring.
\item
$R[x; \alpha]$ is a $CN$ ring.
\end{enumerate}
\end{theorem}

\begin{proof}
As $R$ is $2$-primal, $R/Nil_{*}(R)$ is reduced, and so we have \[J(R[x; \alpha]) = \operatorname{Nil}_*(R[x; \alpha]) = \operatorname{Nil}_*(R)[x; \alpha].\] Hence, points (iii) and (iv) are equivalent.
	
\((i) \Rightarrow (ii)\). Supposing that \( f = \sum_{i=0}^{n} a_i x^i \in R[x; \alpha]\), we have $a_i=c_i+n_i$, where $c_i\in C(R)$ and $n_i\in Nil(R)=Nil_*(R)$ for every $i$. Moreover, we know that $n_0\in Nil(R)\subseteq Nil_*(R)\subseteq J(R)\subseteq \Delta(R)$. Therefore, taking into account Proposition \ref{4.12}, we write $$f=\sum_{i=0}^{n} c_i x^i+n_0+\sum_{i=1}^{n} n_i x^i\in C(R)+\Delta(R)+Nil_*(R)[x; \alpha]x\subseteq C(R)+\Delta(R[x; \alpha])\subseteq$$
$$C(R[x; \alpha])+\Delta(R[x; \alpha]),$$ as required.

\((ii) \Rightarrow (i)\). For every $a\in R$, we have that $ax=\sum_{i=0}^{n} c_i x^i+\sum_{i=0}^{n} d_i x^i$. Thus, $a=c_0+d_0$, where $d_0\in Nil_*(R)$ owing to Proposition \ref{4.12}. Now, it suffices to demonstrate that $c_0\in C(R)$. For this aim, we set $f:=\sum_{i=0}^{n} c_i x^i$. Since we know $rf=fr$ for every $r\in R$, it gives that $rc_0=c_0r$. So, $c_0\in C(R)$ and hence $R$ is a $CN$ ring, as required.

\((i) \iff (iv)\). With a proof very close to the previous implication, the desired result can easily be obtained.
\end{proof}

As a direct consequence, we yield:

\begin{corollary}
Let $R$ be a $2$-primal ring. Then, the following are equivalent:
\begin{enumerate}
\item
$R$ is a $CN$ ring.
\item
$R[x]$ is a $C\Delta$ ring.
\item
$R[x]$ is a $CJ$ ring.
\item
$R[x]$ is a $CN$ ring.
\end{enumerate}
\end{corollary}

Let $Nil^*(R)$ be the sum of all nil-ideals in $R$. We call $Nil^*(R)$ the {\it upper nil-radical} of $R$. This is the largest nil-ideal of $R$. In order to remember the famous {K\"othe}\text{conjecture} to the non-expert reader, we refer to \cite [10.28]{17} for a more detailed information.

\medskip

So, we are in a position to show validity of the following major assertion.

\begin{theorem}\label{4.11}
If $R[x]$ is a $C\Delta$ ring, then $R$ satisfies the {K\"othe}~\text{conjecture}.
\end{theorem}

\begin{proof}
Assume that $R[x]$ is $C\Delta$. It suffices to establish that $J(R[x])=Nil^*(R)[x]$. Indeed, since $J(R[x])=N[x]$, where $N=J(R[x])\cap R$, we show that $Nil^*(R)=N$. To that goal, given $a\in Nil^*(R)$, there exists $k\in \mathbb{N}$ such that $a^k=0$; thus, $a^k\in J(R[x])$. Since $R[x]$ is $C\Delta$, the quotient $R[x]/J(R[x])$ is reduced, which fact is enabled from Proposition \ref{4.3}(vii). So, we conclude that $a\in J(R[x])$ and this insures that $a\in N$. On the other hand, since $N$ is a nil-ideal and $Nil^*(R)$ is the largest nil-ideal of $R$, we infer $N\subseteq Nil^*(R)$. Therefore, $N=Nil^*(R)$, as required.
\end{proof}

The next two necessary and sufficient conditions are worthy of documentation.

\begin{proposition}\label{3.1}
For any ring \( R \) and an integer \( n \geq 1 \), the following are equivalent:
\begin{enumerate}
\item
\( R \) is \(C\Delta\).
\item
\( D_n(R) \) is \(C\Delta\).
\item
\( V_n(R) \) is \(C\Delta\).
\item
\( D_n^k(R) \) is \(C\Delta\).
\item
\( V_n^k(R) \) is \(C\Delta\).
\end{enumerate}
\end{proposition}

\begin{proof}
\((i) \Rightarrow (ii)\). Let
\[
A = \begin{pmatrix} a & & & & * \\ & a & & & \\ & & \ddots & & \\ 0 & & & & a \end{pmatrix} \in D_n(R).
\]
Thanks to (i), there exist \( c \in C(R) \) and \( b \in \Delta(R) \) such that \( a = c + b \). Let
\[
C = \begin{pmatrix} c &   &  & 0 \\  & c &  &  \\  &  & \ddots &  \\  0&  &  & c \end{pmatrix}, \quad B = \begin{pmatrix} b &  &  & * \\  & b &  &  \\  &  & \ddots &  \\ 0 &  &  & b \end{pmatrix}.
\]
Then, we verify at once that \( C \in C(D_n(R)) \) and \( B \in \Delta(D_n(R)) \). Also, we have \( A = C + B \). Therefore, \( D_n(R) \) is \(C\Delta\).

\((ii) \Rightarrow (i)\). Let \( a \in R \). Consider the matrix
\[
A = \begin{pmatrix} a &  &  & 0 \\  & a &  &  \\ \ &  & \ddots &  \\ 0 &  &  & a \end{pmatrix} \in D_n(R).
\]
So, \( A \) has a \(C\Delta\)-decomposition \( A = C + B \) in \(D_n(R)\), where a readily verification gives that
\[
C = \begin{pmatrix} c &  &  & 0 \\  & c &  &  \\  &  & \ddots &  \\ 0 &  &  & c \end{pmatrix} \in C(P_n(R)), \quad B = \begin{pmatrix} b &  &  & * \\  & b &  &  \\  &  & \ddots &  \\ 0 &  &  & b \end{pmatrix} \in \Delta(R_n).
\]
Then, we may write \( a = c + b \) with \( c \in C(R) \) and \( b \in \Delta(R) \). Hence, \( a \) has a \(C\Delta\)-decomposition in \( R \), and so \( R \) is \(C\Delta\).
The equivalencies \((i) \Leftrightarrow (iii)\), \((i) \Leftrightarrow (iv)\) and \((i) \Leftrightarrow (v)\) can be proved similarly, so we eliminate their evidences.
\end{proof}

\begin{proposition}\label{3.2}
Let $R=\prod_{i \in I}R_i$ be a direct product of rings. Then, $R$ is $C\Delta$ if, and only if, $R_i$ is $C\Delta$ for each index $i$.
\end{proposition}

\begin{proof}
We know that \(C(R) = \prod_{i \in I} C(R_i)\) and, thanks to \cite{2}, that \(\Delta(R) = \prod_{i \in I} \Delta(R_i)\). So, our further arguments are routine and thus omitted.
\end{proof}

We noticed earlier that, if \(n \geq 2\), then \(M_n(R)\) is {\it not} \(C\Delta\). We now exhibit certain subrings of \(M_3(R)\) that are \(C\Delta\) whenever \(R\) is \(C\Delta\).

And so, let \(R\) be a ring, and assume \(s, t \in C(R)\). Write
\[
L_{(s,t)} = \left\{ \begin{pmatrix} a & 0 & 0 \\ s c & d & t e \\ 0 & 0 & f \end{pmatrix} \in M_3(R) : a, c, d, e, f \in R \right\},
\]
where the operations are defined as those in \(M_3(R)\). Then, a simple inspection is a guarantor that \(L_{(s,t)}(R)\) is a subring of \(M_3(R)\).

\medskip

We, thereby, arrive at the following technicality.

\begin{lemma}\label{3.3}
Let \(R\) be a ring, and \(s, t \in C(R)\). Then,
\[
\Delta(L_{(s,t)}(R)) = \left\{ \begin{pmatrix} a & 0 & 0 \\ s c & d & t e \\ 0 & 0 & f \end{pmatrix} \in L_{(s,t)}(R) : a, d, f \in \Delta(R), c, e \in R \right\}.
\]
\end{lemma}

\begin{proof}
Assume that \(a, d, f \in \Delta(R)\). Put
\[
A := \begin{pmatrix} a & 0 & 0 \\ s c & d & t e \\ 0 & 0 & f \end{pmatrix} \in L_{(s,t)}(R).
\]
For any element
\[
B = \begin{pmatrix} x & 0 & 0 \\ s y & z & t u \\ 0 & 0 & v \end{pmatrix} \in U(L_{(s,t)}(R)),
\]
we have
\[
I_3 + AB = \begin{pmatrix} 1+ax & 0 & 0 \\ s(cx+by) & 1+dz & t(du+ev) \\ 0 & 0 & 1+fv \end{pmatrix}.
\]
Under our hypothesis, it must be that \(1+ax\), \(1+dz\), \(1+fv \in U(R)\). Hence, \(I_3 + AB \in U(L_{(s,t)}(R))\) and \(A \in \Delta(L_{(s,t)}(R))\), as needed.

Conversely, suppose that \(a, d, f \in R\) and put
\[
A := \begin{pmatrix} a & 0 & 0 \\ sc & d & te \\ 0 & 0 & f \end{pmatrix} \in \Delta(L_{(s,t)}(R)).
\]
For any \(x, z, v \in U(R)\) and \(y, u \in R\), letting
\[
B = \begin{pmatrix} x & 0 & 0 \\ sy & z & tu \\ 0 & 0 & v \end{pmatrix} \in U(L_{(s,t)}(R)),
\]
we have \(I_3 + AB \in U(L_{(s,t)}(R))\), i.e.,
\[
I_3 + AB = \begin{pmatrix} 1+ax & 0 & 0 \\ s(cx+dy) & 1+dz & t(dx+tev) \\ 0 & 0 & 1+fv \end{pmatrix} \in U(L_{(s,t)}(R)).
\]
Thus, we have \(1+ax\), \(1+dz\), \(1+fv \in U(R)\). This unambiguously shows that \(a, d, f \in \Delta(R)\), as required.
\end{proof}

Consider the following subring of \(L_{(s,t)}(R)\):
\[
V_2(L_{(s,t)}(R)) = \left\{ \begin{pmatrix} a & 0 & 0 \\ 0 & a & te \\ 0 & 0 & a \end{pmatrix} \in L_{(s,t)}(R) : a, e \in R \right\}.
\]

\medskip

We are now ready to establish the following.

\begin{proposition}\label{3.4}
Let \(R\) be a ring. Then, the following statements hold:
\begin{enumerate}
\item
\(R\) is a \(C\Delta\) ring if, and only if, \(V_2(L_{(s,t)}(R))\) is a \(C\Delta\) ring.
\item
Assume that \(R\) is a \(C\Delta\) ring. Suppose also that any \(\{a, d, f\} \subseteq R\) having a \(C\Delta\)-decomposition \(a = x + p\), \(d = y + q\) and \(f = z + r\) with \(\{x, y, z\} \subseteq C(R)\) and \(\{p, q, r\} \subseteq \Delta(R)\) satisfy \(sx = sy\) and \(ty = tz\). Then, \(L_{(s,t)}(R)\) is a \(C\Delta\) ring.
\end{enumerate}
\end{proposition}

\begin{proof}
\((i)\) Assume that \(R\) is a \(C\Delta\) ring. Set
\[
A := \begin{pmatrix} a & 0 & 0 \\ 0 & a & t e \\ 0 & 0 & a \end{pmatrix} \in V_2(L_{(s,t)}(R)).
\]
Then, there exist \(c \in C(R)\) and \(r \in \Delta(R)\) such that \(a = c + r\). Hence, for the next two matrices
\[
C = \begin{pmatrix} c & 0 & 0 \\ 0 & c & 0 \\ 0 & 0 & c \end{pmatrix}, \quad B = \begin{pmatrix} r & 0 & 0 \\ 0 & r & t e \\ 0 & 0 & r \end{pmatrix},
\]
we detect that \(C \in C(V_2(L_{(s,t)}(R)))\) and \(B \in \Delta(V_2(L_{(s,t)}(R)))\). Therefore, \(A = C + B\) is the wanted \(C\Delta\)-decomposition of \(A\) in \(V_2(L_{(s,t)}(R))\).

For the converse, choose \(x \in R\) and consider
\[
A := \begin{pmatrix} x & 0 & 0 \\ 0 & x & 0 \\ 0 & 0 & x \end{pmatrix} \in V_2(L_{(s,t)}(R)).
\]
Then, there exist
\[
C = \begin{pmatrix} c & 0 & 0 \\ 0 & c & 0 \\ 0 & 0 & c \end{pmatrix} \in C(V_2(L_{(s,t)}(R))), \quad B = \begin{pmatrix} r & 0 & 0 \\ 0 & r & 0 \\ 0 & 0 & r \end{pmatrix} \in \Delta(V_2(L_{(s,t)}(R))).
\]
So, \(c \in C(R)\) and \(r \in \Delta(R)\), and \(x= c + r\) is the \(C\Delta\)-decomposition of \(x\). Consequently, \(R\) is a \(C\Delta\) ring, as asserted.

(ii) Suppose that
\[
A := \begin{pmatrix} a & 0 & 0 \\ s c & d & t e \\ 0 & 0 & f \end{pmatrix} \in L_{(s,t)}(R).
\]
Let \(a = x + p\), \(d = y + q\) and \(f = z + r\) denote the \(C\Delta\)-decompositions of \(a\), \(d\), and \(f\), respectively. By our hypothesis, \(sx = sy\) and \(ty = tz\). Thus, we deduce that \(A\) has a \(C\Delta\)-decomposition in \(L_{(s,t)}(R)\) as \(A = C + B\), where
\[
C = \begin{pmatrix} x & 0 & 0 \\ 0 & y & 0 \\ 0 & 0 & z \end{pmatrix} \in C(L_{(s,t)}(R)), \quad B = \begin{pmatrix} p & 0 & 0 \\ s c & q & t e \\ 0 & 0 & r \end{pmatrix} \in \Delta(L_{(s,t)}(R)),
\]
as claimed.
\end{proof}

As an automatic consequence, we derive:

\begin{corollary}\label{3.5}
Let \(R\) be a ring. If \(L_{(s,t)}(R)\) is a \(C\Delta\) ring, then \(R\) is a \(C\Delta\) ring.
\end{corollary}

\begin{proof}
Assume that \(L_{(s,t)}(R)\) is a \(C\Delta\) ring, and let \(a \in R\) and
\[
A := \begin{pmatrix} a & 0 & 0 \\ 0 & a & 0 \\ 0 & 0 & a \end{pmatrix} \in L_{(s,t)}(R).
\]
Bearing in mind the hypothesis, \(A\) has a \(C\Delta\)-decomposition in \(L_{(s,t)}(R)\). This easily leads to the fact that \(a\) admits a \(C\Delta\)-decomposition in \(R\), as necessary.
\end{proof}

The next example shows that there are \(C\Delta\) rings such that \(L_{(s,t)}(R)\) need {\it not} be a \(C\Delta\) ring.

\begin{example}\label{3.6}
Let \(R = \mathbb{Z}\) and \(A := \begin{pmatrix} 1 & 0 & 0 \\ 2 & 2 & 1 \\ 0 & 0 & 1 \end{pmatrix} \in L_{(1,1)}(R)\). Since \(R\) is a commutative ring, it is too a \(C\Delta\) ring. Assume that \(A = C + B\) is a \(C\Delta\)-decomposition of \(A\). Since \(A\) is neither a central element nor an element in the \(\Delta\), we can assume with no harm in generality that \(A\) has a \(C\Delta\)-decomposition as \(A = C + B\), where
\[
C := \begin{pmatrix} 1 & 0 & 0 \\ 0 & 1 & 0 \\ 0 & 0 & 1 \end{pmatrix} \in C(L_{(1,1)}(R)) \quad \text{and} \quad B := \begin{pmatrix} x & 0 & 0 \\ c & y & c\\ 0 & 0 & z \end{pmatrix} \in \Delta(L_{(1,1)}(R)),
\]
where \(\{x, y, z\} \subseteq \Delta(Z)\). This, however, forces a contradiction with the arithmetic in \(\mathbb{Z}\), as expected.
\end{example}

The next two affirmations are also worthy of recording.

\begin{proposition}\label{3.7}
A ring \(R\) is \(C\Delta\) if, and only if, so is \(L_{(0,0)}(R)\).
\end{proposition}

\begin{proof}
Firstly, observe that \(L_{(0,0)}(R)\) is isomorphic to the ring \(R \times R \times R\). Thus, employing Proposition \ref{3.2}, \(\prod_{i \in I} R_i\) is a \(C\Delta\) ring for each \(i \in I\) with all $R_i\cong R$ whenever $R$ is a \(C\Delta\) ring.
\end{proof}

Furthermore, suppose now that \(R \) is a ring and that \(s, t \in C(R)\). Let \(H_{(s,t)}(R)\) be the subring of \(M_3(R)\) defined by
\[
H_{(s,t)}(R) = \left\{ \begin{pmatrix} a & 0 & 0 \\ s c & d &  e \\ 0 & 0 & f \end{pmatrix} \in M_3(R) : a, c, d, e, f \in R, a - d = sc, d - f = te \right\}.
\]
Thus, an easy inspection gives that any element \(A\) of \(H_{(s,t)}(R)\) has the form
\[
A = \begin{pmatrix} sc+te+f & 0 & 0 \\ c & te+f & e \\ 0 & 0 & f \end{pmatrix}.
\]

\begin{lemma}\label{3.8}
Let \(R\) be a ring and \(s, t \in C(R)\). Then,
\[
\Delta(H_{(s,t)}(R)) = \left\{ \begin{pmatrix} a & 0 & 0 \\ c & d & e \\ 0 & 0 & f \end{pmatrix} \in H_{(s,t)}(R) : a, d, f \in \Delta(R), c, e \in R \right\}.
\]
\end{lemma}

\begin{proof}
Assume that \(a, d, f \in \Delta(R)\), and set
\[
A := \begin{pmatrix} a & 0 & 0 \\ c & d & e \\ 0 & 0 & f \end{pmatrix} \in H_{(s,t)}(R)
\]
as well as, for any \(x, z, v \in U(R)\); \(y, u \in R\), set
\[
B := \begin{pmatrix} x & 0 & 0 \\ y & z & u \\ 0 & 0 & v \end{pmatrix} \in U(H_{(s,t)}(R)).
\]
Then, it follows that
\[
I_3 + AB = \begin{pmatrix} 1+ax & 0 & 0 \\ y & 1+dz & u \\ 0 & 0 & 1+fv \end{pmatrix}.
\]
But, since \(1+ax\), \(1+dz\), \(1+fv \in U(R)\), we can get \(I_3 + AB \in U(H_{(s,t)}(R))\) and this yields that \(A \in \Delta(H_{(s,t)}(R))\).

Conversely, suppose that
\[
A := \begin{pmatrix} a & 0 & 0 \\ c & d & e \\ 0 & 0 & f \end{pmatrix} \in \Delta(H_{(s,t)}(R)).
\]
and, for any \(x, z, v \in U(R)\); \(y, u \in R\), set
\[
B := \begin{pmatrix} x & 0 & 0 \\ y & z & u \\ 0 & 0 & v \end{pmatrix} \in U(H_{(s,t)}(R)).
\]
Then, it follows that
\[
I_3 + AB = \begin{pmatrix} 1+ax & 0 & 0 \\ y & 1+dz & u \\ 0 & 0 & 1+fv \end{pmatrix} \in U(H_{(s,t)}(R)).
\]
Therefore, we obtain \(1+ax\), \(1+dz\), \(1+fv \in U(R)\), and hence \(a, d, f \in \Delta(R)\), as wanted.
\end{proof}

Our next chief affirmation is the following one.

\begin{theorem}\label{3.9}
A ring \(R\) is \(C\Delta\) if, and only if, \(H_{(s,t)}(R)\) is a \(C\Delta\) ring.
\end{theorem}

\begin{proof}
Assume that \(R\) is a \(C\Delta\) ring. Put
\[
A := \begin{pmatrix} a & 0 & 0 \\ c & d & e \\ 0 & 0 & f \end{pmatrix} \in H_{(s,t)}(R).
\]
Thus, one follows that \(a = c_1 + r_1\), \(d = c_2 + r_2\), \(f = c_3 + r_3\), \(c = c_4 + r_4\), \(e = c_5 + r_5\) with \(\{c_1, c_2, c_3, c_4, c_5\} \subseteq C(R)\) and \(\{r_1, r_2, r_3, r_4, r_5\} \subseteq \Delta(R)\). Letting \(c_1 - c_2 = sc_4\), \(c_2 - c_3 = tc_5\), \(r_1 - r_2 = sr_4\) and \(r_2 - r_3 = tr_5\), we then construct
\[
C := \begin{pmatrix} c_1 & 0 & 0 \\ c_4 & c_2 & c_5 \\ 0 & 0 & c_3 \end{pmatrix} \quad \text{and} \quad B := \begin{pmatrix} r_1 & 0 & 0 \\ r_4 & r_2 & r_5 \\ 0 & 0 & r_3 \end{pmatrix}.
\]
So, we receive \(C \in C(H_{(s,t)}(R))\) and \(B \in \Delta(H_{(s,t)}(R))\). Hence, \(A = C + B\) is the desired \(C\Delta\)-decomposition of \(A\).

Conversely, suppose that \(H_{(s,t)}(R)\) is a \(C\Delta\) ring and choose \(a \in R\). Then,
\[
A := \begin{pmatrix} a & 0 & 0 \\ 0 & a & 0 \\ 0 & 0 & a \end{pmatrix} \in H_{(s,t)}(R),
\]
and it has the \(C\Delta\)-decomposition \(A = C + B\), where
\[
C = \begin{pmatrix} x & 0 & 0 \\ y & z & u \\ 0 & 0 & v \end{pmatrix} \in C(H_{(s,t)}(R)) \quad \text{and} \quad B = \begin{pmatrix} r_1 & 0 & 0 \\ r_2 & r_3 & r_4 \\ 0 & 0 & r_5 \end{pmatrix} \in \Delta(H_{(s,t)}(R)),
\]
with \(\{y, u, v\} \subseteq C(R)\) and \(\{r_1, r_2, r_3\} \subseteq \Delta(R)\). Consequently, \(a = v + r_5\) is the asked \(C\Delta\)-decomposition of the element \(a\).
\end{proof}

Given a ring $R$ and a central elements $s$ of $R$, the $4$-tuple $\begin{pmatrix}
	R & R\\
	R & R
\end{pmatrix}$ becomes a ring denoted by $K_{s}(R)$ with addition component-wise and with multiplication defined by
$$\begin{pmatrix}
	a_{1} & x_{1}\\
	y_{1} & b_{1}
\end{pmatrix}\begin{pmatrix}
	a_{2} & x_{2}\\
	y_{2} & b_{2}
\end{pmatrix}=\begin{pmatrix}
	a_{1}a_{2}+sx_{1}y_{2} & a_{1}x_{2}+x_{1}b_{2} \\
	y_{1}a_{2}+b_{1}y_{2} & sy_{1}x_{2}+b_{1}b_{2}
\end{pmatrix}.$$
Then, $K_{s}(R)$ is called a {\it generalized matrix ring} over $R$.

\medskip

We now proceed by proving the following claim.

\begin{proposition}\label{3.10}
Let \(R\) be a commutative ring. Then, the following hold:
\begin{enumerate}
\item
\(C(K_0(R))\) consists of all scalar matrices.
\item
 \(U(K_0(R)) = \{\begin{pmatrix} a & b \\ c & d \end{pmatrix} \in K_0(R) : a, d \in U(R)\}\).
\item
\(\Delta(K_0(R)) = \{\begin{pmatrix} a & b \\ c & d \end{pmatrix} \in K_0(R) : a, d \in \Delta(R)\}\).
\end{enumerate}
\end{proposition}

\begin{proof}
\begin{enumerate}
\item
It follows at once from \cite[Lemma 1.1]{4}.
\item
It follows at once from \cite[Lemma 3.1]{5}.
\item
Assume that \(a, d \in \Delta(R)\), and let
	\[
	A = \begin{pmatrix} a & b \\ c & d \end{pmatrix} \in K_0(R).
	\]
	For any element
	\[
	B = \begin{pmatrix} a^\prime & b^\prime \\ c^\prime & d^\prime \end{pmatrix} \in U(K_0(R)),
	\]
	we compute
	\[
	I_2 + AB = \begin{pmatrix} 1+aa' & ab' + bd' \\ ca' + dc' & 1+dd' \end{pmatrix}.
	\]
	By hypothesis, \(1+aa'\), \(1+dd' \in U(R)\). Thus, \(I_2 + AB \in U(K_0(R))\) implying that \(A \in \Delta(K_0(R))\).

Conversely, suppose that \(a, d \in R\) and
	\[
	A = \begin{pmatrix} a & b \\ c & d \end{pmatrix} \in \Delta(K_0(R)).
	\]
	For any \(a', d' \in U(R)\) and \(c', b' \in R\), we set
	\[
	B ;= \begin{pmatrix} a' & b' \\ c' & d' \end{pmatrix} \in U(K_0(R)).
	\]
We, furthermore, calculate that
\[
I_2 + AB = \begin{pmatrix} 1+aa' & ab'+bd' \\ ca'+dc' & 1+dd' \end{pmatrix} \in U(K_0(R)).
\]
So, we have \(1+aa'\), \(1+dd'\) \(\in U(R)\) showing that \(a, d \in \Delta(R)\), as required.
\end{enumerate}
\end{proof}

\begin{proposition}\label{3.11}
A ring \(R\) is \(C\Delta\) if, and only if, \(D_n(K_0(R))\) is a \(C\Delta\) ring for any $n\geq 2$.
\end{proposition}

\begin{proof}
Let \(R\) be a \(C\Delta\) ring. We assume for a moment that \(n = 2\). Suppose
\[
A := \begin{pmatrix} a & b \\ 0 & a \end{pmatrix} \in D_2(K_0(R)).
\]
By assumption, \(a = c_1 + r_1\), where \(c_1 \in C(R)\) and \(r_1 \in \Delta(R)\). Suppose also that
\[
C := \begin{pmatrix} c_1 & 0 \\ 0 & c_1 \end{pmatrix} \in C(D_2(K_0(R))) \quad \text{and} \quad B := \begin{pmatrix} r_1 & b\\ 0 & r_1 \end{pmatrix} \in \Delta(D_2(K_0(R))).
\]
Hence, \(A = C + B\) is the required \(C\Delta\)-decomposition of the matrix \(A\).

Conversely, let \(D_2(K_0(R))\) be a \(C\Delta\) ring and \(a \in R\). Then,
\[
A := \begin{pmatrix} a & 0 \\ 0 & a \end{pmatrix} \in D_2(K_0(R))
\]
has the \(C\Delta\)-decomposition \(A = C + B\) with
\[
C := \begin{pmatrix} c_1 & 0 \\ 0 & c_1 \end{pmatrix} \in C(D_2(K_0(R))) \quad \text{and} \quad B := \begin{pmatrix} r_1 & b_1 \\ 0 & r_1 \end{pmatrix} \in \Delta(D_2(K_0(R))),
\]
where \(c_1 \in C(R)\) and \(r_1 \in \Delta(R)\). By comparing components of matrices, we get \(a = c_1 + r_1\), and obviously it is the required \(C\Delta\)-decomposition of the element \(a\).

The general case for any $n>2$ follows by induction, which we leave to the interested reader for a direct check.
\end{proof}

\begin{example}\label{3.12}
Note that \(K_0(R)\) need {\it not} be a \(C\Delta\) ring. In fact, assume the contrary that
\[
A := \begin{pmatrix} 1 & 0 \\ 0 & 0 \end{pmatrix} \in K_0(\mathbb{Z})
\]
has a \(C\Delta\)-decomposition as \(A = C + B\), where \(c \in C(K_0(\mathbb{Z}))\) and \(B \in \Delta(K_0(\mathbb{Z}))\). Then, we infer
\[
C = \begin{pmatrix} x & 0 \\ 0 & x \end{pmatrix} \quad \text{and} \quad B = \begin{pmatrix} 1-x & 0 \\ 0 & -x\end{pmatrix},
\]
and these, in turn, imply \(x = 1\) or \(x \in \Delta(\mathbb{Z}) = \{0\}\), that is, the expected contradiction, thus substantiating our claim.
\end{example}

Let $R$ be a ring and $M$ a bi-module over $R$. The {\it trivial extension} of $R$ and $M$ is stated as
\[ T(R, M) = \{(r, m) : r \in R \text{ and } m \in M\}, \]
with addition defined component-wise and multiplication defined by
\[ (r, m)(s, n) = (rs, rn + ms). \]
The trivial extension $T(R, M)$ is isomorphic to the subring
\[ \left\{ \begin{pmatrix} r & m \\ 0 & r \end{pmatrix} : r \in R \text{ and } m \in M \right\} \]
of the formal $2 \times 2$ matrix ring $\begin{pmatrix} R & M \\ 0 & R \end{pmatrix}$, and also $T(R, R) \cong R[x]/\left\langle x^2 \right\rangle$. We, likewise, note that the set of units of the trivial extension $T(R, M)$ is
\[ U(T(R, M)) = T(U(R), M). \]
Moreover, according to \cite {3}, we have
\[ \Delta(T(R, M)) = T(\Delta(R), M). \]

So, the following is valid.

\begin{proposition}\label{3.15}
If \(T(R,M)\) is a \(C\Delta\) ring, then \(R\) is a \(C\Delta\) ring. The converse holds if, for any \(c \in C(R)\) and \(m \in M\), we have \(cm = mc\).
\end{proposition}

\begin{proof}
Assuming that \(T(R,M)\) is a \(C\Delta\) ring, it easily follows that $R$ is \(C\Delta\) too.

Conversely, let \(R\) be a \(C\Delta\) ring and let \(X = (a,m) \in T(R,M)\) be an arbitrary element. Then, \(a \in R\) and, by hypothesis, we write \(a = c + r\), where \(c \in C(R)\) and \(r \in \Delta(R)\). Thus, \[X = (a,m) = (c+r,m) = (c,0) + (r,m),\] where \((c,0) \in C(T(R,M))\) and \(T(\Delta(R),M) = \Delta(T(R,M))\). Putting \(C = (c,0)\) and \(D = (r,m)\), we conclude that \(X = C + D\) is a \(C\Delta\)-decomposition of \(X\), as we need.
\end{proof}

As an immediate consequence, we receive:

\begin{corollary}\label{3.16}
The trivial extension $T(R, R)$ is a $C\Delta$ ring if, and only if, $R$ is a $C\Delta$ ring.	
\end{corollary}

Let $\alpha$ be an endomorphism of $R$ and $n$ a positive integer. It was defined by Nasr-Isfahani in \cite{1} the {\it skew triangular matrix ring} like this:

$${\rm T}_{n}(R,\alpha )=\left\{ \left. \begin{pmatrix}
	a_{0} & a_{1} & a_{2} & \cdots & a_{n-1} \\
	0 & a_{0} & a_{1} & \cdots & a_{n-2} \\
	0 & 0 & a_{0} & \cdots & a_{n-3} \\
	\ddots & \ddots & \ddots & \vdots & \ddots \\
	0 & 0 & 0 & \cdots & a_{0}
\end{pmatrix} \right| a_{i}\in R \right\}$$
with addition point-wise and multiplication given by
\begin{align*}
	&\begin{pmatrix}
		a_{0} & a_{1} & a_{2} & \cdots & a_{n-1} \\
		0 & a_{0} & a_{1} & \cdots & a_{n-2} \\
		0 & 0 & a_{0} & \cdots & a_{n-3} \\
		\ddots & \ddots & \ddots & \vdots & \ddots \\
		0 & 0 & 0 & \cdots & a_{0}
	\end{pmatrix}\begin{pmatrix}
		b_{0} & b_{1} & b_{2} & \cdots & b_{n-1} \\
		0 & b_{0} & b_{1} & \cdots & b_{n-2} \\
		0 & 0 & b_{0} & \cdots & b_{n-3} \\
		\ddots & \ddots & \ddots & \vdots & \ddots \\
		0 & 0 & 0 & \cdots & b_{0}
	\end{pmatrix}  =\\
	& \begin{pmatrix}
		c_{0} & c_{1} & c_{2} & \cdots & c_{n-1} \\
		0 & c_{0} & c_{1} & \cdots & c_{n-2} \\
		0 & 0 & c_{0} & \cdots & c_{n-3} \\
		\ddots & \ddots & \ddots & \vdots & \ddots \\
		0 & 0 & 0 & \cdots & c_{0}
	\end{pmatrix},
\end{align*}
where $$c_{i}=a_{0}\alpha^{0}(b_{i})+a_{1}\alpha^{1}(b_{i-1})+\cdots +a_{i}\alpha^{i}(b_{0}),~~ 1\leq i\leq n-1.$$ We denote the elements of ${\rm T}_{n}(R, \alpha)$ by $(a_{0},a_{1},\ldots , a_{n-1})$. If $\alpha $ is the identity endomorphism, then ${\rm T}_{n}(R,\alpha )$ is a subring of upper triangular matrix ring ${\rm T}_{n}(R)$.

\medskip

The following criterion is true:

\begin{proposition}\label{3.17}
Let $R$ be a ring. Then, the following are equivalent:
\begin{enumerate}
\item
$R$ is a $C\Delta$ ring.
\item
${\rm T}_{n}(R,\alpha )$ is a $C\Delta$ ring.
\end{enumerate}
\end{proposition}

\begin{proof}
\((ii) \Rightarrow (i)\). Choose
	$$I:=\left\{
	\left.
	\begin{pmatrix}
		0 & a_{12} & \ldots & a_{1n} \\
		0 & 0 & \ldots & a_{2n} \\
		\vdots & \vdots & \ddots & \vdots \\
		0 & 0 & \ldots & 0
	\end{pmatrix} \right| a_{ij}\in R \quad (i\leq j )
	\right\}.$$
Therefore, one easily finds that $I^{n}=(0)$, $I\subseteq J({\rm T}_{n}(R,\alpha ))$, and $\dfrac{{\rm T}_{n}(R,\alpha )}{I} \cong R$. Resultantly, the implication follows from Proposition \ref{4.3}(i).\\

\((i) \Rightarrow (ii)\). Consider $A:=\begin{pmatrix}
	a_{0} & a_{1} & a_{2} & \cdots & a_{n-1} \\
	0 & a_{0} & a_{1} & \cdots & a_{n-2} \\
	0 & 0 & a_{0} & \cdots & a_{n-3} \\
	\ddots & \ddots & \ddots & \vdots & \ddots \\
	0 & 0 & 0 & \cdots & a_{0}
\end{pmatrix} \in{\rm T}_{n}(R,\alpha )$, where $a_{i}\in R$. Then, by hypothesis, $a_{0}=c_{0}+r_{0}$, where $c_{0}\in C(R)$ and $r_{0}\in \Delta(R)$. Therefore, we write $A=C+B$, where $C:=\begin{pmatrix}
c_{0} & 0 & 0 & \cdots & 0 \\
0 & c_{0} & 0 & \cdots & 0 \\
0 & 0 & c_{0} & \cdots & 0 \\
\ddots & \ddots & \ddots & \vdots & \ddots \\
0 & 0 & 0 & \cdots & c_{0}\end{pmatrix}\in C({\rm T}_{n}(R,\alpha ))$, and $B:=\begin{pmatrix}
r_{0} & a_{1} & a_{2} & \cdots & a_{n-1} \\
0 & r_{0} & a_{1} & \cdots & a_{n-2} \\
0 & 0 & r_{0} & \cdots & a_{n-3} \\
\ddots & \ddots & \ddots & \vdots & \ddots \\
0 & 0 & 0 & \cdots & r_{0}
\end{pmatrix} \in \Delta({\rm T}_{n}(R,\alpha ))$ looking at \cite [Corollary 9(1)]{2}. Hence, $T_{n}(R,\alpha )$ is a $C\Delta$ ring, as claimed.
\end{proof}

Besides, notice that there is a ring isomorphism $$\varphi : \dfrac{R[x;\alpha]}{\langle x^{n}\rangle }\rightarrow {\rm T}_{n}(R,\alpha),$$ defined as $$\varphi (a_{0}+a_{1}x+\ldots +a_{n-1}x^{n-1}+\langle x^{n} \rangle )=(a_{0},a_{1},\ldots ,a_{n-1})$$ with $a_{i}\in R$, $0\leq i\leq n-1$. Thus, one finds that ${\rm T}_{n}(R,\alpha )\cong \dfrac{R[x;\alpha ]}{\langle  x^{n}\rangle}$, where $\langle x^{n}\rangle$ is the ideal generated by $x^{n}$.

\medskip

We arrive in the sequel at a series of consequences.

\begin{corollary}\label{3.18}
Let $R$ be a ring with an endomorphism $\alpha$ such that $\alpha (1)=1$. Then, the following are equivalent:
\begin{enumerate}
\item
$R$ is a $C\Delta$ ring.
\item
$\dfrac{R[x;\alpha ]}{\langle x^{n}\rangle }$ is a $C\Delta$ ring.
\item
$\dfrac{R[[x;\alpha ]]}{\langle x^{n}\rangle }$ is a $C\Delta$ ring.
\end{enumerate}
\end{corollary}

\begin{corollary}\label{3.19}
Let $R$ be a ring. Then, the following are equivalent:
\begin{enumerate}
\item
$R$ is a $C\Delta$ ring.
\item
$\dfrac{R[x]}{\langle x^{n}\rangle }$ is a $C\Delta$ ring.
\item
$\dfrac{R[[x]]}{\langle x^{n}\rangle }$ is a $C\Delta$ ring.
\end{enumerate}
\end{corollary}

Consider $R$ to be a ring and $M$ to be a bi-module over $R$. Let $${\rm DT}(R,M) := \{ (a, m, b, n) | a, b \in R, m, n \in M \}$$ with addition defined componentwise and multiplication defined by $$(a_1, m_1, b_1, n_1)(a_2, m_2, b_2, n_2) = (a_1a_2, a_1m_2 + m_1a_2, a_1b_2 + b_1a_2, a_1n_2 + m_1b_2 + b_1m_2 +n_1a_2).$$ Then, ${\rm DT}(R,M)$ is a ring which is isomorphic to ${\rm T}({\rm T}(R, M), {\rm T}(R, M))$. Likewise, with this at hand, we deduce $${\rm DT}(R, M) =
\left\{\begin{pmatrix}
	a &m &b &n\\
	0 &a &0 &b\\
	0 &0 &a &m\\
	0 &0 &0 &a
\end{pmatrix} |  a,b \in R, m,n \in M\right\}.$$ We now have the following isomorphism of rings: $\dfrac{R[x, y]}{\langle x^2, y^2\rangle} \rightarrow {\rm DT}(R, R)$ given as $$a + bx + cy + dxy \mapsto
\begin{pmatrix}
	a &b &c &d\\
	0 &a &0 &c\\
	0 &0 &a &b\\
	0 &0 &0 &a
\end{pmatrix}.$$

\begin{corollary}
Let $R$ be a ring. Then, the following statements are equivalent:
\begin{enumerate}
\item
$R$ is a $C\Delta$ ring.
\item
${\rm DT}(R, R)$ is a $C\Delta$ ring.
\item
$\dfrac{R[x, y]}{\langle x^2, y^2\rangle}$ is a $C\Delta$ ring.
\end{enumerate}
\end{corollary}

\begin{corollary}\label{3.20}
Let
$ R $
be a ring, and let
\begin{center}
		$S_{n}(R):=\left\lbrace (a_{ij})\in T_{n}(R)\, | \, a_{11}=a_{22}=\cdots=a_{nn}\right\rbrace.$
\end{center}
Then, the following are equivalent:
\begin{enumerate}
\item	
$R$ is a $C\Delta$ ring.
\item	
$S_{n}(R)$ is a $C\Delta$ ring.
\end{enumerate}
\end{corollary}

Further, Wang introduced in \cite{7} the matrix ring $S_{n,m}(R)$. Supposing $R$ is a ring, then the matrix ring $S_{n,m}(R)$ can be represented as

$$\left\{ \begin{pmatrix}
	a & b_1 & \cdots & b_{n-1} & c_{1n} & \cdots & c_{1 n+m-1}\\
	\vdots  & \ddots & \ddots & \vdots & \vdots & \ddots & \vdots \\
	0 & \cdots & a & b_1 & c_{n-1,n} & \cdots & c_{n-1,n+m-1} \\
	0 & \cdots & 0 & a & d_1 & \cdots & d_{m-1} \\
	\vdots  & \ddots & \ddots & \vdots & \vdots & \ddots & \vdots \\
	0 & \cdots & 0 & 0  & \cdots & a & d_1 \\
	0 & \cdots & 0 & 0  & \cdots & 0 & a
\end{pmatrix}\in T_{n+m-1}(R) : a, b_i, d_j,c_{i,j} \in R \right\}.$$

In this vein, let $T_{n,m}(R)$ be

$$\left\{ \left(\begin{array}{@{}c|c@{}}
	\begin{matrix}
		a & b_1 & b_2 & \cdots & b_{n-1} \\
		0 & a & b_1 & \cdots & b_{n-2} \\
		0 & 0 & a & \cdots & b_{n-3} \\
		\vdots & \vdots & \vdots & \ddots & \vdots \\
		0 & 0 & 0 & \cdots & a
	\end{matrix}
	& \bigzero \\
	\hline
	\bigzero &
	\begin{matrix}
		a & c_1 & c_2 & \cdots & c_{m-1} \\
		0 & a & c_1 & \cdots & c_{m-2} \\
		0 & 0 & a & \cdots & c_{m-3} \\
		\vdots & \vdots & \vdots & \ddots & \vdots \\
		0 & 0 & 0 & \cdots & a
	\end{matrix}
\end{array}\right)\in T_{n+m}(R) : a, b_i,c_j \in R \right\}, $$

\noindent and

$$U_{n}(R)=\left\{ \begin{pmatrix}
	a & b_1 & b_2 & b_3 & b_4 & \cdots & b_{n-1} \\
	0 & a & c_1 & c_2 & c_3 & \cdots & c_{n-2} \\
	0 & 0 & a & b_1 & b_2 & \cdots & b_{n-3} \\
	0 & 0 & 0 & a & c_1 & \cdots & c_{n-4} \\
	\vdots & \vdots & \vdots & \vdots &  &  & \vdots \\
	0 &0 & 0 & 0 & 0 & \cdots & a
\end{pmatrix}\in T_{n}(R) :  a, b_i, c_j \in R \right\}.$$

\begin{corollary}\label{3.21}
Let $R$ be a ring. Then, the following issues are equivalent:
\begin{enumerate}
\item
$R$ is a $C\Delta$ ring.
\item
$S_{n,m}(R)$ is a $C\Delta$ ring.
\item
$T_{n,m}(R)$ is a $C\Delta$ ring.
\item
$U_{n}(R)$ is a $C\Delta$ ring.
\end{enumerate}	
\end{corollary}

Letting now $R$ be an arbitrary ring, Danchev et al. introduced in \cite{8} the aforementioned rings as follows:
\begin{align*}
	&A_{n,m}(R) =R[x,y | x^n=xy=y^m=0], \\
	&B_{n,m}(R) =R\left\langle x,y | x^n=xy=y^m=0  \right\rangle, \\
	&C_{n}(R) =R \langle x,y | x^2=\underbrace{xyxyx...}_{\text{$n-1$ words}}=y^2=0 \rangle.
\end{align*}

\begin{lemma}\label{3.22}~(\cite[Lemma 5.1]{8})
Let $R$ be a ring and $m, n \in \mathbb{N}$. Then, the next three isomorphisms of rings are fulfilled:
\begin{enumerate}
\item	
$A_{n,m}(R) \cong T_{n,m}(R)$.
\item	
$B_{n,m}(R) \cong S_{n,m}(R)$.
\item	
$C_{n}(R) \cong U_{n}(R)$.
\end{enumerate}
\end{lemma}

\begin{corollary}\label{3.23}
Let $R$ be a ring. Then, the following claims are equivalent:
\begin{enumerate}
\item
$R$ is a $C\Delta$ ring.
\item
$B_{n,m}(R)$ is a $C\Delta$ ring.
\item
$A_{n,m}(R)$ is a $C\Delta$ ring.
\item
$C_{n}(R)$ is a $C\Delta$ ring.
\end{enumerate}	
\end{corollary}

\section{Open questions}

We finish our considerations with two left problems as follows:

\begin{problem}\label{p1}
Let $R$ be a ring and $G$ a multiplicative group. Find a necessary and sufficient condition when the group ring $RG$ is $C\Delta$ only in terms of $R$, $G$ and their sections.	
\end{problem}

\begin{problem}\label{p2}
Characterize those rings $R$, naming them {\it feebly $\Delta$-clean} for which, for any $r\in R$, there exist $d\in \Delta(R)$ and $e, f\in Id(R)$ with $ef=fe=0$ such that $r=d+e-f$. In addition, if $d$ commutes with (either) $e$ or $f$, the feebly $\Delta$-clean ring $R$ having this property is said to be {\it strongly feebly $\Delta$-clean}. 	
\end{problem}


\noindent{\bf Funding:} The work of the first-named author, P.V. Danchev, is partially supported by the Junta de Andaluc\'ia under Grant FQM 264. All other three authors are supported by Bonyad-Meli-Nokhbegan and receive funds from this foundation.


\vskip3.0pc

\begin{thebibliography}{99}	
	
\bibitem{18}
W. Chen, {\it On constant products of elements in skew polynomial rings}, Bull. Iranian Math. Soc. {\bf 41}(2) (2015), 453--462.	
	
\bibitem{8}
P. Danchev, A. Javan, O. Hasanzadeh and A. Moussavi, {\it Rings with $u-1$ quasinilpotent for each unit u}, J. Algebra Appl. {\bf 24}(10) (2025).

\bibitem{13}
P. V. Danchev and T. Y. Lam, {\it Rings with unipotent units}, Publ. Math. (Debrecen) {\bf 88}(3-4) (2016), 449--466.	
	
\bibitem{3}
F. Karabacak, M. Kosan, T. Quynh and D. Yai, {\it A generalization of UJ-rings}, J. Algebra Appl. {\bf 20}(12) (2021).

\bibitem{12}
M. T. Kosan, A. Leroy and J. Matczuk, {\it On UJ-rings}, Commun. Algebra 46(5) (2018), 2297--2303.

\bibitem{4}
P. A. Krylov and A. A. Tuganbaev, {\it Modules over formal matrix rings}, J. Math.Sci. 171(2) (2010), 248-295.

\bibitem{6}
Y. Kurtulmaz, S. Halicioglu, A. Harmanci and H. Chen, {\it Rings in which elements are a sum of a central and a unit element}, Bull. Belg. Math. Soc. Simon Stevin {\bf 42}(4) (2019), 619--631.

\bibitem{11}
Y. Kurtulmaz and A. Harmanci, {\it Rings in which elements are a sum of a central and a nilpotent element}, arXiv:2005.12575 (2020).

\bibitem{17}
T. Y. Lam, {\it A First Course in Noncommutative Rings}, Graduate Texts in Mathematics {\bf 131}, Springer-Verlag, New York (1991).

\bibitem{10}
T. Y. Lam, {\it Exercises in Classical Ring Theory}, Second Edition, Springer Verlag, New York (2003).

\bibitem{2}
A. Leroy and J. Matczuk, {\it Remarks on the Jacobson radical}, in "Rings, Modules and Codes", 269--276; Contemp. Math. {\bf 727}, Amer. Math. Soc., Providence, RI (2019).	

\bibitem{16}
J. Levitzki, {\it On the structure of algebraic algebras and related rings}, Trans. Am. Math. Soc. {\bf 74} (1953), 384--409.

\bibitem{9}
G. Ma, Y. Wang and A. Leroy, {\it Rings in which elements are sum of a central element and an element in the Jacobson radical}, Czechoslovak Mathematical Journal. {\bf 74}(2) (2024), 515--533.
	
\bibitem{1}
A. R. Nasr-Isfahani, {\it On skew triangular matrix rings}, Commun. Algebra. {\bf 39}(11) (2011), 4461--4469.

\bibitem{nic}
W. K. Nicholson, {\it Lifting idempotents and exchange rings}, Trans. Am. Math. Soc. {\bf 229} (1977) 269--278.	

\bibitem{5}
G. Tang, C. Li and Y. Zhou, {\it Study of Morita contexts}, Commun. Algebra. {\bf 42}(4) (2014), 1668--1681.	

\bibitem{7}
W. Wang, E. R. Puczylowski and L. Li, {\it On Armendariz rings and matrix rings with simple 0-multiplication}, Commun. Algebra {\bf 36}(4) (2008), 1514--1519.

\bibitem{15}
Z. L. Ying, T. Kosan and Y. Zhou, {\it Rings in which every element is a sum of two tripotents}, Canad. Math. Bull. {\bf 59} (2016) 661-–672.

\end{thebibliography}
\end{document}